\documentclass{siamart}

\usepackage{amsfonts,amssymb,mathtools}

\usepackage{microtype}

\usepackage{xcolor}
\usepackage[percent]{overpic}

\definecolor{cgblue}{RGB}{0,114,178}
\definecolor{cgorange}{RGB}{213,94,0}

\newcommand{\curveanno}[2]{%
  \begingroup\setlength{\fboxsep}{1.0pt}%
  \colorbox{white}{\strut\scriptsize\color{#1}\ensuremath{#2}}%
  \endgroup}

\makeatletter
\AtBeginDocument{%
 \def\refstepcounter@optarg[#1]#2{%
 \cref@old@refstepcounter{#2}%
 \cref@constructprefix{#2}{\cref@result}%
 \@ifundefined{cref@#1@alias}%
 {\def\@tempa{#1}}%
 {\def\@tempa{\csname cref@#1@alias\endcsname}}%
 \protected@edef\cref@currentlabel{%
 [\@tempa][\arabic{#2}][\cref@result]%
 \csname p@#2\endcsname\csname the#2\endcsname}}%
}
\makeatother

\newcommand{\R}{\mathbb R}
\newcommand{\ip}[2]{\left\langle #1,#2\right\rangle}
\newcommand{\supp}{\operatorname{supp}}
\newcommand{\dist}{\operatorname{dist}}
\newcommand{\coeff}[2]{[x^{#1}]#2}
\newcommand{\norm}[1]{\lVert #1\rVert}

\newcommand{\Pe}{P_{\mathrm e}}
\newcommand{\Po}{P_{\mathrm o}}
\newcommand{\cC}{\mathcal W_{\mathrm{cyc}}}
\newcommand{\Ce}{\mathcal W_{\mathrm e}}
\newcommand{\Co}{\mathcal W_{\mathrm o}}
\newcommand{\Omegaey}{\Omega_{\mathrm e}^{y}}
\newcommand{\Omegaoy}{\Omega_{\mathrm o}^{y}}
\newcommand{\Omegaew}{\Omega_{\mathrm e}^{w}}
\newcommand{\Omegaow}{\Omega_{\mathrm o}^{w}}
\newcommand{\retE}{\operatorname{ret}_E}

\newcommand{\TheTitle}{A Proof of the Forsythe Conjecture for the Two-Step Restarted Conjugate Gradient Method}
\newcommand{\TheAuthors}{Matthew J. Colbrook, George Stepaniants, and Alex Townsend}
\newcommand{\TheShortAuthors}{M. J. Colbrook, G. Stepaniants, and A. Townsend}
\headers{Two-Step Restarted CG}{\TheShortAuthors}
\title{\TheTitle}
\author{
Matthew J. Colbrook\thanks{Department of Applied Mathematics and Theoretical Physics, University of Cambridge, Cambridge CB3 0WA, UK (\email{m.colbrook@damtp.cam.ac.uk}).}
\and
George Stepaniants\thanks{Department of Computing and Mathematical Sciences, California Institute of Technology, Pasadena, CA 91125, USA (\email{gstepan@caltech.edu}).}
\and
Alex Townsend\thanks{Department of Mathematics, Cornell University, Ithaca, NY 14853, USA (\email{townsend@cornell.edu}).}
}

\ifpdf
\hypersetup{hidelinks,pdftitle={\TheTitle},pdfauthor={\TheAuthors}}
\fi

\begin{document}
\maketitle

\begin{abstract}
Forsythe conjectured in 1968 that the normalized residual directions of restarted conjugate gradients approach a two-cycle. We prove the conjecture for restart length two: unless the method terminates, its even and odd normalized residuals converge separately.
\end{abstract}



\section{Introduction}

The conjugate gradient (CG) method is a standard Krylov method for symmetric positive definite linear systems~\cite{HestenesStiefel1952,LiesenStrakos2013}. Classical convergence theory controls errors and residual norms, but it does not determine the asymptotic behavior of residual directions. This distinction becomes important under restarting since the residual norm can tend to zero while the normalized residuals have their own behavior. Whether these residual directions approach a two-cycle has been a longstanding theoretical question about the dynamics of restarted CG. The analysis can be used to develop an extrapolated version for potentially faster convergence (see Section~\ref{sec:extrapolation}).

Forsythe's analysis first circulated as Stanford technical report CS-TR-67-61, dated April 13, 1967, and appeared in journal form in 1968~\cite{Forsythe1967,Forsythe1968}. He studied the $s$-dimensional optimum gradient method, equivalently CG restarted every $s$ steps~\cite{FaberLiesenTichy2023,Forsythe1968}. For $s=1$, this iteration is steepest descent. Forsythe and Motzkin conjectured its two-direction asymptotics in 1951, and Akaike proved the conjecture in 1959 for $s=1$~\cite{Akaike1959,Forsythe1968,ForsytheMotzkin1951}. Later work refined this asymptotic picture in finite and Hilbert spaces and studied the associated gradient-norm oscillations~\cite{GonzagaSchneider2016,NocedalSartenaerZhu2002,PronzatoWynnZhigljavsky2001,PronzatoWynnZhigljavsky2006}; a broader review of spectral gradient dynamics appears in~\cite{ZouMagoules2022}. Guided by the $s=1$ theorem and by numerical evidence for $s=2$, Forsythe conjectured that, for every fixed restart length, either the iteration terminates or the even and odd normalized residuals converge separately. Their limits form an asymptotic two-cycle, even though the unnormalized residuals tend to zero.

The Russian original of Zhuk and Bondarenko's paper appeared in 1983, and its English translation appeared in 1984~\cite{ZhukBondarenko1984}. The paper explicitly claims the Forsythe conjecture for $s=2$ and proves convergence of the associated polynomial coefficients. Unfortunately, a particular part of the proof that takes coefficient convergence to a limiting residual direction invokes a statement attributed to Zabolotskaya~\cite[p.~238]{Zabolotskaya1979}. Faber, Liesen, and Tich\'y observed that the English translation states the needed implication only conditionally and does not prove it; they therefore continue to regard the two-step case as open~\cite[sec.~4.1, especially pp.~12--13]{FaberLiesenTichy2023}. Zhuk later proved the existence of limiting iteration parameters for the $s$-step method when minimizing a quadratic functional in Hilbert space~\cite{Zhuk1995}; this parameter-convergence result does not by itself establish convergence of residual directions for $s=2$. A complementary dynamical-system analysis, including a detailed numerical study of $s=2$, appears in~\cite{PronzatoWynnZhigljavsky2009}.

As recently as April 2026, an open-problems survey continued to list the Forsythe conjecture, including $s=2$, as open~\cite[sec.~2.7]{AmselEtAl2026}. To the best of our knowledge, the argument below gives the first self-contained proof of the $s=2$ directional-convergence statement, closing the gap identified by Faber, Liesen, and Tich\'y.

\subsection{Sketch of the proof}
In an eigenbasis of $A$, the squared coordinates of a normalized residual form a probability vector $w_k$, and each restart block is governed by a monic quadratic $P_k$. The energy identities in \cref{lem:energy-displacement} give $H_k\uparrow h$ and $\|y_{k+2}-y_k\|\to0$. At a subsequential limit, the product of the two limiting quadratics minus $h$ is a monic quartic vanishing on the active eigenvalues. Positive energy excludes fewer than three active nodes, while the quartic degree excludes more than four (\cref{prop:localization}). The double-orthogonality identity~\eqref{eq:double-orth} then turns energy increments into coefficient changes, yielding parity limits for $P_k$ and a common quartic limit $q_\infty$ (\cref{prop:global-lipschitz}).

If all limit points use three nodes, a support and its limiting quadratic determine at most one weight vector. Each parity limit set is therefore finite and, because it is connected, a singleton (\cref{prop:three-node}). If a four-node limit exists, $q_\infty(\lambda)=h$ selects a fixed four-node set $E$, and the compatible exact two-cycles form a closed one-parameter family with three-node endpoints (\cref{lem:edge-family}). Coefficient convergence shows that the trajectory approaches this family, but does not rule out drift along it.

To exclude such drift, we first show that the total weight outside $E$ decays geometrically. Indeed, every surviving external coordinate is contractive apart from the possible multiplier $-1$, which the logarithmic phase argument in \cref{prop:phase} excludes; see~\eqref{eq:external-decay}. After renormalizing on $E$, a scalar $\alpha$ parameterizes position along the family. The restricted-dynamics identity~\eqref{eq:face-motion} and sign transport in \cref{lem:factor-sharing} give its motion a fixed orientation, while the endpoint-uniform retraction estimate~\eqref{eq:retraction} adds only an exponentially small error. The perturbed-orientation argument in \cref{lem:perturbed-orientation} gives finite total variation of $\alpha$; \cref{prop:frame-convergence} then gives convergence of the weights, and \cref{thm:main} restores the coordinate signs.

Forsythe established energy monotonicity, parity asymptotic regularity, connected invariant limit sets, localization to supports of size $s+1$ through $2s$, and exact two-cycles on $s+1$ nodes~\cite[Thms.~3.3, 3.8, 4.7, and 4.8]{Forsythe1968}. Zhuk and Bondarenko proved qualitative convergence of the monic two-block quartic coefficients for $s=2$~\cite[pp.~429--432]{ZhukBondarenko1984}, and Pronzato, Wynn, and Zhigljavsky developed the probability-measure dynamics and three- and four-point invariant geometry~\cite[secs.~1.2--1.4]{PronzatoWynnZhigljavsky2009}. The remaining task is passage from these polynomial and limit-set facts to one limiting direction, especially endpoint-uniform control of drift. We give a self-contained proof to reconcile the different notations and supply that control.

Sections~2--4 develop the spectral maps, global estimates, and limiting geometry; Section~5 proves convergence; and Section~6 discusses extrapolation, longer restart lengths, and the role of the three-node endpoints.

\section{Two-step restarted CG and spectral dynamics}

The conjugate gradient method was developed by Hestenes and Stiefel, while Lanczos developed a closely related iteration at about the same time~\cite{GolubOLeary1989,HestenesStiefel1952,Lanczos1952}. In exact arithmetic, CG terminates in at most $d$ steps on a $d\times d$ matrix. Its later use as an iteration for large sparse systems (instead of a direct method) was brought to renewed prominence by Reid~\cite{Reid1971}; Golub and O'Leary describe Reid's paper as a catalyst for much of the subsequent work~\cite[p.~56]{GolubOLeary1989}. We first recall ordinary CG, then formulate fixed two-step restarting and derive the normalized spectral map analyzed throughout the paper.
\subsection{The ordinary conjugate gradient method}

Let $A\in\R^{d\times d}$ be symmetric positive definite, set $x_*=A^{-1}b$, and let $r_0=b-Ax_0$. After $m$ ordinary CG steps, $r_m=\pi_m(A)r_0$, where $\deg\pi_m\leq m$ and $\pi_m(0)=1$. Equivalently, CG has the variational characterization~\cite{HestenesStiefel1952,LiesenStrakos2013}
$$
x_m\in x_0+\mathcal K_m(A,r_0),\qquad
\norm{x_*-x_m}_A
=\min_{z\in x_0+\mathcal K_m(A,r_0)}\norm{x_*-z}_A,
$$
where $\mathcal K_m(A,r)=\operatorname{span}\{r,Ar,\ldots,A^{m-1}r\}$ and $\norm{u}_A^2=u^\top Au$.
This is the characterization needed for restarting: each block computes the best approximation in a fixed low-dimensional Krylov space.

\subsection{Two-step restarting}
Hestenes and Stiefel already observed that CG could be started anew from a current approximation, principally to limit the effects of roundoff~\cite[pp.~409--410]{HestenesStiefel1952}. The fixed-length iteration studied here is Forsythe's optimum $s$-gradient method in modern notation: each outer step computes the exact CG minimizer over an $s$-dimensional Krylov space and then restarts from the resulting residual~\cite{Forsythe1967,Forsythe1968}.

We use $k$ for the outer restart-block index, with approximation $x^{[k]}$ and residual $r_k=b-Ax^{[k]}$ at the start of block $k$. An $s$-step restarted cycle satisfies
\begin{equation}\label{eq:restarted-CG}
x^{[k+1]}\in x^{[k]}+\mathcal K_s(A,r_k),
\qquad
\norm{x_*-x^{[k+1]}}_A
=\min_{z\in x^{[k]}+\mathcal K_s(A,r_k)}
 \norm{x_*-z}_A.
\end{equation}
This is a modern formulation of Forsythe's $s$-dimensional optimum gradient method~\cite{FaberLiesenTichy2023,Forsythe1968}. For $s=2$, it means performing two ordinary CG steps from $x^{[k]}$ and restarting from the resulting residual; in particular, $x^{[k+1]}-x^{[k]}\in\operatorname{span}\{r_k,Ar_k\}$.

The restarted error contracts at least as much as one steepest-descent step per block. Since the trial space in~\eqref{eq:restarted-CG} contains the steepest-descent line, the classical Kantorovich estimate gives, unless termination occurs~\cite{Kantorovich1948,LiesenStrakos2013},
$$
\norm{x_*-x^{[k+1]}}_A
\leq\frac{\kappa(A)-1}{\kappa(A)+1}\norm{x_*-x^{[k]}}_A,
\qquad
\kappa(A)=\frac{\lambda_{\max}(A)}{\lambda_{\min}(A)}.
$$
Hence $x^{[k]}\to x_*$ and $r_k\to0$, although this says nothing by itself about the directions $y_k=r_k/\norm{r_k}$. Throughout, we consider the nonterminating case, so $r_k\neq0$ for every $k$.

Multiplying the incoming residual by a nonzero scalar leaves the CG coefficients unchanged and multiplies the outgoing residual by the same scalar. Thus, whenever a two-step block does not terminate, its normalized outgoing residual depends only on its normalized incoming residual. We denote this one-block map on unit directions by $F$: $F(y)$ is the outgoing residual normalized to unit length when the incoming direction is $y$. Along a nonterminating restarted orbit, $y_{k+1}=F(y_k)$; the explicit spectral formula appears in~\eqref{eq:F}.

\begin{theorem}[Forsythe's conjecture for restart length two]
\label{thm:main}
If two-step restarted CG does not terminate, then there are unit vectors $y_{\mathrm e}$ and $y_{\mathrm o}$ such that the normalized block-boundary residuals satisfy $y_{2n}\to y_{\mathrm e}$ and $y_{2n+1}\to y_{\mathrm o}$. Moreover, the two limits form a two-cycle for this normalized one-block map: $y_{\mathrm o}=F(y_{\mathrm e})$ and $y_{\mathrm e}=F(y_{\mathrm o})$.
\end{theorem}

\Cref{fig:forsythe-conjecture} illustrates the theorem for a $10\times10$ diagonal matrix: the projected even and odd trajectories approach distinct points, and both full-space subsequences converge. Forsythe reported similar $s=2$ experiments, but did not specify the matrices or initial vectors~\cite[p.~66]{Forsythe1968}.

\begin{figure}[t]
  \centering
  \begin{minipage}[t]{0.485\textwidth}
    \vspace{0pt}\centering
    \begin{overpic}[width=\linewidth]{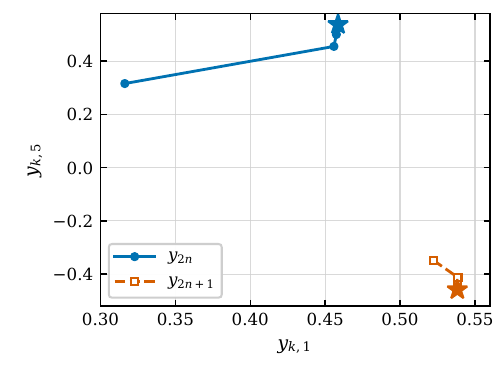}
      \put(72,69){\makebox(0,0)[l]{\curveanno{cgblue}{y_{\mathrm e}}}}
      \put(78,20){\makebox(0,0)[l]{\curveanno{cgorange}{y_{\mathrm o}}}}
    \end{overpic}
  \end{minipage}\hfill
  \begin{minipage}[t]{0.485\textwidth}
    \vspace{0pt}\centering
    \begin{overpic}[width=\linewidth]{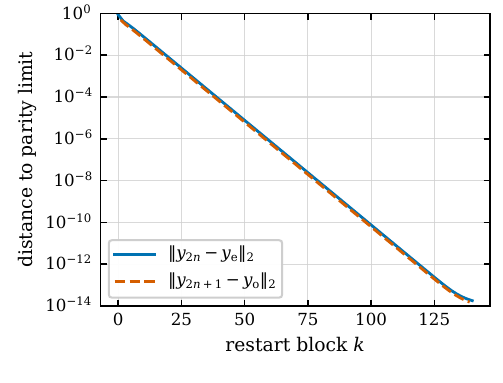}
    \end{overpic}
  \end{minipage}
  \caption{Two-step restarted CG for $A=\operatorname{diag}(1,\ldots,10)$
  with equal initial spectral weights. Left: the even and odd trajectories
  projected onto the $(y_{k,1},y_{k,5})$-plane; the computation uses all ten
  coordinates. Right: full-space distances to the final computed even and
  odd states decay to roundoff, illustrating \cref{thm:main}.}
  \label{fig:forsythe-conjecture}
\end{figure}

\subsection{Spectral coordinates and active eigenvalues}

Working in the fixed spectral decomposition of $A$ makes each restart block act coordinatewise. The coordinate signs are needed to recover the residual direction, but their squares form a probability vector and already determine the CG polynomial. We therefore first study the squared spectral weights and restore the signs at the end.

Let $N\leq d$ be the number of distinct eigenvalues whose spectral projections of $r_0$ are nonzero, and list them as $0<\lambda_1<\cdots<\lambda_N$. We call this an $N$-node state because the spectral dynamics has one coordinate for each of these $N$ active eigenspaces. Let $\Pi_i$ be the spectral projection associated with $\lambda_i$, and define $v_i=\Pi_i r_0/\|\Pi_i r_0\|$. The $v_i$ are fixed orthonormal eigenvectors in the active directions. Since every residual is a polynomial in $A$ applied to $r_0$, its $\lambda_i$-component remains in the span of $v_i$; hence $y_k=\sum_{i=1}^{N}y_{k,i}v_i$. The numbers $w_{k,i}=y_{k,i}^2$ form a probability vector $w_k=(w_{k,1},\dots,w_{k,N})$. For $\lambda=\lambda_i$, we also use $y_{k,\lambda}:=y_{k,i}$ and $w_{k,\lambda}:=w_{k,i}$, with the same convention when the time index is absent. We identify supports with their spectral nodes: $\supp y_k=\supp w_k:=\{\lambda_i:w_{k,i}>0\}$. This spectral probability-measure formulation goes back to Akaike for $s=1$ and was developed for the optimum $s$-gradient iteration by Forsythe and by Pronzato, Wynn, and Zhigljavsky~\cite{Akaike1959,Forsythe1968,PronzatoWynnZhigljavsky2009}.

\subsection{The monic two-step CG polynomial}

Over one two-step block, the outgoing residual has the form $\pi(A)r$, where $\pi$ is the CG residual polynomial, $\deg\pi\leq2$, and $\pi(0)=1$. Because the outgoing residual is normalized, only the polynomial up to an overall scalar matters, except for the sign of that scalar. We choose its monic representative $P_w$. An argument below proves that $P_w(0)>0$, so this change of normalization does not reverse the residual direction. The representation of CG residuals by orthogonal polynomials is classical~\cite[secs.~14--15]{HestenesStiefel1952} and was developed systematically by Stiefel~\cite{Stiefel1958}; Forsythe used the same representation for the optimum $s$-gradient dynamics~\cite[secs.~2--3]{Forsythe1968}.

We now use the standing nontermination assumption. If the incoming residual has components in at most two distinct eigenspaces, its relative minimal polynomial has degree at most two, and exact CG makes the residual zero during that block~\cite{HestenesStiefel1952}. Because no block terminates, every $w_k$ therefore has at least three active nodes~\cite[Thm.~5.1]{Forsythe1968}~\cite[Thm.~2]{PronzatoWynnZhigljavsky2009}. We will nevertheless encounter boundary states in compactness arguments, so it is convenient to define the associated monic orthogonal quadratic for every probability vector with at least two active directions.

For a probability vector $w$ supported on at least two distinct nodes, let $P_w(x)=x^2+a(w)x+b(w)$ be the monic quadratic orthogonal to $1$ and $x$ for the discrete weighted inner product $\langle f,g\rangle_w:=\sum_iw_i f(\lambda_i)g(\lambda_i)$. We use $\|f\|_w^2:=\langle f,f\rangle_w$ for the corresponding weighted polynomial norm. Equivalently, $P_w$ is the degree-two monic orthogonal polynomial for the discrete spectral measure $\sum_iw_i\delta_{\lambda_i}$ and we can write~\cite[sec.~2]{Forsythe1968}~\cite[eqs.~(1.11)--(1.13)]{PronzatoWynnZhigljavsky2009}
\begin{equation}
\sum_iw_iP_w(\lambda_i)=0,
\qquad
\sum_iw_i\lambda_iP_w(\lambda_i)=0.
\label{eq:orthogonality}
\end{equation}
Define $H(w):=\sum_iw_iP_w(\lambda_i)^2$ and, whenever $H(w)>0$, $(Tw)_i:=w_iP_w(\lambda_i)^2/H(w)$.
For a signed unit vector $y=(y_i)$, set $w_i=y_i^2$; whenever $H(w)>0$, define
\begin{equation}
F(y)_i:=\frac{P_w(\lambda_i)y_i}{\sqrt{H(w)}}.
\label{eq:F}
\end{equation}
We use $P_y=P_w$ and $H(y)=H(w)$.
Here $H(w)=\|P_w(A)y\|^2$ is the squared norm of the filtered direction. Thus $T$ evolves the squared spectral weights, while $F$ retains their signs; in particular, $F(y)_i^2=(Tw)_i$.

Existence and uniqueness are standard consequences of the moment formulation~\cite{GolubMeurant2010,LiesenStrakos2013}. We include the short argument because the same Gram matrix controls continuity and uniform invertibility later, and because the least-squares identity is used repeatedly.
For $P_w=x^2+ax+b$, the two conditions in~\eqref{eq:orthogonality} form a linear system for $(b,a)$ with Gram matrix
$$
\Gamma(w):=
\begin{pmatrix}
1&\sum_iw_i\lambda_i\\
\sum_iw_i\lambda_i&\sum_iw_i\lambda_i^2
\end{pmatrix}.
$$
For $z=(u,v)^\top\neq0$, $z^\top\Gamma(w)z=\sum_iw_i(u+v\lambda_i)^2>0$ whenever the support contains at least two distinct nodes. Thus $\Gamma(w)$ is positive definite, so $P_w$ exists and is unique. Orthogonality also gives, for every polynomial $\ell$ with $\deg\ell\leq1$,
\begin{equation}\label{eq:least-squares}
\norm{P_w+\ell}_w^2=H(w)+\norm{\ell}_w^2.
\end{equation}
Hence $P_w$ uniquely minimizes $\sum_iw_iQ(\lambda_i)^2$ over monic quadratics $Q$. Moreover, $H(w)=0$ exactly when $P_w$ vanishes on every active node. On exactly two nodes, $P_w$ is the monic polynomial vanishing at those nodes, so $H(w)=0$; on at least three nodes, a nonzero quadratic cannot vanish on the whole support, so $H(w)>0$. A one-node state terminates after one exact CG step. Consequently, the nontermination hypothesis gives $H(w_k)>0$ for every $k$, so $T$ and $F$ are defined along the entire orbit and any spectral coordinate that vanishes remains zero under both maps.

It remains to fix the sign relating the monic polynomial to the usual residual polynomial. If $\supp w=\{a,b\}$ with $0<a<b$, then the preceding uniqueness argument gives $P_w(x)=(x-a)(x-b)$, and hence $P_w(0)=ab>0$. If $w$ has at least three active nodes, the standard zero-location theorem for orthogonal polynomials associated with a positive measure places the two distinct roots of $P_w$ in the open interval spanned by the support~\cite[sec.~3.3]{Szego1975}. That interval is positive, so again $P_w(0)>0$.

We can now identify the abstract maps $F$ and $T$ with one restarted CG block.
The residual polynomial $\pi_k$ of block $k$ satisfies $\pi_k(0)=1$ and, by the Galerkin condition, $\langle\pi_k,1\rangle_{w_k}=\langle\pi_k,x\rangle_{w_k}=0$. If $\deg\pi_k\leq1$, write $\pi_k=u+vx$. The two Galerkin equations then read $\Gamma(w_k)(u,v)^\top=0$, so positive definiteness forces $\pi_k=0$, contradicting $\pi_k(0)=1$. Hence $\deg\pi_k=2$, and uniqueness gives $\pi_k=P_k/P_k(0)$, where $P_k:=P_{w_k}$. Since $P_k(0)>0$, normalizing the next residual removes this positive scalar, so $y_{k+1}=F(y_k)$ and $w_{k+1}=Tw_k$.

\subsection{The two-block recurrence}

The conjecture concerns even and odd subsequences, so it is natural to compose two consecutive restart filters. For brevity, let $H_k:=H(w_k)=\|P_k(A)y_k\|^2$ and $q_k:=P_kP_{k+1}$. Thus $P_k$ and $q_k$ are the one- and two-block spectral filters, respectively, while $\sqrt{H_k}$ is the one-block normalization factor. Equation~\eqref{eq:F} gives
\begin{equation}\label{eq:signed-ratio}
y_{k+2,i}
=\frac{q_k(\lambda_i)}{\sqrt{H_kH_{k+1}}}y_{k,i}.
\end{equation}
When a single state $w$ is under consideration, we suppress the time index and write $w^+:=Tw$, $P:=P_w$, $P^+:=P_{w^+}=P_{Tw}$, and $q:=PP^+$. Thus a superscript ${}^+$ means one application of $T$, while a subscript $\infty$, as in $q_\infty$, denotes a limiting object.

\section{Global asymptotics and localization}

We now extract the information shared by the even and odd sequences. The energy first identifies all subsequential limits as low-support two-cycles. Double orthogonality then turns the fact that the nonnegative energy increments have finite total sum into convergence of the polynomial coefficients, which places every limiting support in one fixed set of at most four eigenvalues.

\subsection{Energy, displacement, compactness, and localization}

The normalization $H_k$ controls the long-time dynamics: its increment measures the two-block defect and the displacement of the signed direction, yielding monotonicity and two-step asymptotic regularity. Compactness and continuity then localize every $\omega$-limit point to a three- or four-node two-cycle. This combination of monotone energy and two-step asymptotic regularity is due to Forsythe~\cite[Thm.~3.3]{Forsythe1968}; see also~\cite[Thm.~6]{PronzatoWynnZhigljavsky2009}. We record the exact identity in our monic normalization because its telescoping defect is used later.

\begin{lemma}\label{lem:energy-displacement}
For every nonterminating iterate,
\begin{align}
H_{k+1}-H_k
&=\frac{1}{H_{k+1}}\sum_iw_{k,i}\bigl(q_k(\lambda_i)-H_{k+1}\bigr)^2,
\label{eq:variance}\\
\sqrt{H_{k+1}}-\sqrt{H_k}
&=\frac{\sqrt{H_{k+1}}}{2}\|y_{k+2}-y_k\|^2.
\label{eq:displacement}
\end{align}
Consequently, with $d_k:=H_{k+1}-H_k$, $H_k$ tends to $h$ from below for some $h>0$, $\sum_{k=0}^{\infty}d_k<\infty$, and $\sum_{k=0}^{\infty}\|y_{k+2}-y_k\|^2<\infty$.
\end{lemma}

\begin{proof}
Since $P_{k+1}-P_k$ has degree at most one, orthogonality gives $\langle P_k,P_{k+1}-P_k\rangle_{w_k}=0$. Hence $\langle P_k,P_{k+1}\rangle_{w_k}=H_k$ and $\sum_iw_{k,i}q_k(\lambda_i)=H_k$. The weight update also gives $\sum_iw_{k,i}q_k(\lambda_i)^2=H_k\sum_iw_{k+1,i}P_{k+1}(\lambda_i)^2=H_kH_{k+1}$. Expanding the square in~\eqref{eq:variance} proves the identity.

Equation~\eqref{eq:signed-ratio} and $\sum_iw_{k,i}q_k(\lambda_i)=H_k$ give $\ip{F^2(y_k)}{y_k}=\sqrt{H_k/H_{k+1}}$. Since both vectors are unit vectors, $\|F^2(y_k)-y_k\|^2=2-2\sqrt{H_k/H_{k+1}}$, which rearranges to~\eqref{eq:displacement}.

Choosing the monic quadratic $x^2$ in the least-squares characterization gives $0<H_k\le\sum_iw_{k,i}\lambda_i^4\le\lambda_N^4$. Equation~\eqref{eq:variance} makes $H_k$ nondecreasing; hence $H_k\to h\ge H_0>0$, and $\sum_kd_k$ telescopes. Finally,~\eqref{eq:displacement} and $H_{k+1}\ge H_0$ give $\sum_{k=0}^{M}\|y_{k+2}-y_k\|^2\le2(\sqrt{H_{M+1}}-\sqrt{H_0})/\sqrt{H_0}$, uniformly in $M$.
\end{proof}

Although $\sum_k\|y_{k+2}-y_k\|^2<\infty$, this does not imply $\sum_k\|y_{k+2}-y_k\|<\infty$. Thus two-step asymptotic regularity alone does not imply convergence, and we must study subsequential limits.

For a sequence $(x_n)$ in a metric space, write $\omega(x_n)$ for the set of its subsequential limits, equivalently $\bigcap_{m\ge0}\overline{\{x_n:n\ge m\}}$. For the signed directions and squared weights, respectively, set $\Omegaey=\omega(y_{2n})$, $\Omegaoy=\omega(y_{2n+1})$, $\Omegaew=\omega(w_{2n})$, and $\Omegaow=\omega(w_{2n+1})$. Then
$$
\Omegaew=\bigl\{(y_i^2)_{i=1}^N:y\in\Omegaey\bigr\},
\qquad
\Omegaow=\bigl\{(y_i^2)_{i=1}^N:y\in\Omegaoy\bigr\}.
$$
Indeed, coordinatewise squaring is continuous, so it maps each signed limit set into the corresponding weight limit set. Conversely, if $w_{2n_j}\to w^*$, compactness of the unit sphere gives a further subsequence with $y_{2n_j}\to y^*$, and then $(y_i^*)^2=w_i^*$. The odd case is identical.

We use two elementary compactness observations. First, a sequence in a compact metric space eventually enters every open neighborhood of its $\omega$-limit set; otherwise a subsequence outside that neighborhood would have an $\omega$-limit point there. Second, if additionally $\dist(x_{n+1},x_n)\to0$, then $\omega(x_n)$ is nonempty, compact, and connected. The second assertion is the classical connected-cluster-set theorem of Ostrowski~\cite[Thm.~28.1]{Ostrowski1966}; we include the short compact-metric proof. If the limit set split into two nonempty compact sets at positive distance, then once the jumps were smaller than that distance, every passage between neighborhoods of the two sets would contain a point outside both; a convergent subsequence of such passage points would give an $\omega$-limit point outside the alleged union. Applying this observation separately to the even and odd subsequences of signed directions shows that $\Omegaey$ and $\Omegaoy$ are compact and connected, and continuity of coordinatewise squaring gives the same conclusion for $\Omegaew$ and $\Omegaow$.

We also record the continuity needed below. If $w^*$ has at least three positive coordinates, those coordinates remain positive in a sufficiently small relative neighborhood. The determinant of $\Gamma(w)$ therefore stays positive, so inversion of the moment system shows that $w\mapsto P_w$ is continuous there. It follows that $H$ is continuous and bounded away from zero nearby, and the formulas defining $T$ and $F$ make both maps continuous at the corresponding weight and signed states.

Positive energy and the least-squares characterization exclude terminating supports; two-step asymptotic regularity and continuity then make every remaining limit a genuine two-cycle. The result is the $s=2$ specialization of Forsythe's limit-set and support theorems~\cite[Thms.~3.8 and 4.7]{Forsythe1968}; modern formulations appear in~\cite[Thms.~3 and 5]{FaberLiesenTichy2023} and~\cite[Thms.~7 and 8]{PronzatoWynnZhigljavsky2009}.

\begin{proposition}\label{prop:localization}
Every $\omega$-limit point $y^*$ of either the even or the odd subsequence satisfies $F^2(y^*)=y^*$. If $S=\supp y^*$, then $3\le |S|\le4$, $H(y^*)=H(F(y^*))=h$, and $P_{y^*}(\lambda_i)P_{F(y^*)}(\lambda_i)=h$ for $\lambda_i\in S$.
\end{proposition}

\begin{proof}
No point $w^*\in\Omegaew\cup\Omegaow$ can be supported on at most two nodes. Indeed, if $w_{k_j}\to w^*$ with $|\supp w^*|\le2$, choose a fixed monic quadratic $Q$ vanishing on $\supp w^*$. Equation~\eqref{eq:least-squares} gives $0\le H_{k_j}\le\sum_iw_{k_j,i}Q(\lambda_i)^2\to0$, contradicting $H_{k_j}\ge H_0>0$.

Now choose a convergent subsequence $y_{k_j}\to y^*$ whose indices are all even or all odd. By the continuity just established, $y_{k_j+1}=F(y_{k_j})\to F(y^*)$. This is an $\omega$-limit point of the other subsequence, so it also has at least three active nodes. Therefore $F$ is continuous at the intermediate point $F(y^*)$, and $y_{k_j+2}=F^2(y_{k_j})\to F^2(y^*)$. On the other hand, \cref{lem:energy-displacement} gives $\|y_{k_j+2}-y_{k_j}\|\to0$. Hence $F^2(y^*)=y^*$.

Continuity along the selected subsequence and its one-step image gives $H(y^*)=\lim_jH(y_{k_j})=h$ and $H(F(y^*))=\lim_jH(y_{k_j+1})=h$. Applying~\eqref{eq:F} twice and using $F^2(y^*)=y^*$ gives, for every active coordinate, $y_i^*=P_{y^*}(\lambda_i)P_{F(y^*)}(\lambda_i)y_i^*/h$. Canceling $y_i^*\neq0$ gives the product identity. The monic quartic $P_{y^*}P_{F(y^*)}-h$ vanishes on $S$, so $|S|\le4$.
\end{proof}

\subsection{Convergence of the quadratic factors and their products}

\Cref{prop:localization} describes each subsequential limit, but it does not rule out different subsequences approaching different two-cycles. To compare all such limits, we now prove that the monic restart polynomials have unique limits along the even and odd subsequences and that the two-block products $q_k=P_kP_{k+1}$ converge to a single quartic. The key is the following double-orthogonality identity. Together with uniform invertibility of the late moment Gram matrices, it bounds changes in the polynomial coefficients by the energy increments $d_k=H_{k+1}-H_k$. Since $\sum_kd_k=h-H_0<\infty$, these coefficient changes can be controlled after telescoping, without dividing by individual spectral weights, which may tend to zero. Set $\Delta_k(x):=q_k(x)-H_{k+1}$.

For every polynomial $\varphi$ of degree at most one, double orthogonality gives
\begin{equation}
\langle P_{k+2}-P_k,\varphi\rangle_{w_{k+2}}
=-\frac{1}{H_kH_{k+1}}
\langle P_k\varphi,\Delta_k^2\rangle_{w_k}.
\label{eq:double-orth}
\end{equation}
To verify this identity, use the orthogonality of $P_{k+2}$ to $\varphi$ for $w_{k+2}$ to write the left side as $-\langle P_k,\varphi\rangle_{w_{k+2}}$. By the two-block recurrence, this equals $-(H_kH_{k+1})^{-1}\langle P_k\varphi,(H_{k+1}+\Delta_k)^2\rangle_{w_k}$. The constant term vanishes because $P_k\perp\varphi$, and the linear term also vanishes because
\begin{align*}
\langle P_k\varphi,\Delta_k\rangle_{w_k}
&=\langle P_k\varphi,q_k\rangle_{w_k}-H_{k+1}\langle P_k,\varphi\rangle_{w_k}\\
&=\sum_iw_{k,i}P_k(\lambda_i)^2P_{k+1}(\lambda_i)\varphi(\lambda_i)
=H_k\langle P_{k+1},\varphi\rangle_{w_{k+1}}=0.
\end{align*}
Only the quadratic term remains, proving~\eqref{eq:double-orth}.

Qualitative convergence of the monic two-block quartic coefficients for $s=2$ was established by Zhuk and Bondarenko~\cite[pp.~429--432]{ZhukBondarenko1984}, and Zhuk later proved convergence of the iteration parameters along the even and odd subsequences for general $s$ in a Hilbert space~\cite[Thm.~2]{Zhuk1995}. We give the finite-dimensional estimate below because the quantitative tail bound in~\eqref{eq:q-limit-tail} is needed later in the phase analysis.

An unsubscripted $C>0$ denotes a finite constant, uniform in the variables under discussion, whose value may be enlarged from one estimate to the next. Within a single theorem or lemma statement, all appearances of $C$ refer to the same constant.

\begin{proposition}\label{prop:global-lipschitz}
There exist a constant $C<\infty$, monic quadratics $\Pe,\Po$, and a monic quartic $q_\infty$ such that
\begin{align}
\|P_{k+2}-P_k\|&\le Cd_k,
&\|q_{k+1}-q_k\|&\le Cd_k,
\label{eq:factor-bounds}\\
P_{2n}&\to\Pe,
&P_{2n+1}&\to\Po, \label{eq:even-odd-P}\\
q_k&\to q_\infty=\Pe\Po,
&\|q_k-q_\infty\|&\le C(h-H_k).
\label{eq:q-limit-tail}
\end{align}
The inequalities hold for all sufficiently large $k$. All polynomial norms are Euclidean norms of coefficient vectors.
\end{proposition}

\begin{proof}
Let $\Omega^w:=\omega(w_k)=\Omegaew\cup\Omegaow$. By \cref{prop:localization}, every $w\in\Omega^w$ has at least three positive coordinates, and hence $\Gamma(w)$ is positive definite. The function $w\mapsto\lambda_{\min}(\Gamma(w))$ is continuous and strictly positive on the compact set $\Omega^w$. Therefore
$$
\gamma:=\min_{w\in\Omega^w}\lambda_{\min}(\Gamma(w))>0.
$$
The relative open set $\{w:\lambda_{\min}(\Gamma(w))>\gamma/2\}$ is a neighborhood of $\Omega^w$. By the compactness observation above, it contains $w_k$ for every sufficiently large $k$. If $P_k(x)=x^2+a_kx+b_k$, then
$$
\Gamma(w_k)\binom{b_k}{a_k}
=-\binom{\sum_iw_{k,i}\lambda_i^2}{\sum_iw_{k,i}\lambda_i^3}.
$$
The right-hand side is uniformly bounded because the spectral nodes are fixed, while $\norm{\Gamma(w_k)^{-1}}\leq2/\gamma$. Thus the coefficients of $P_k$ are uniformly bounded on the late tail.

Choose $k_0$ so that the preceding coefficient bound holds for $k\ge k_0$. Let $D_k=P_{k+2}-P_k$, so $\deg D_k\le1$, and define $B_P:=\max_{\varphi\in\{1,x\}}\sup_{\substack{k\ge k_0\\1\le i\le N}}|P_k(\lambda_i)\varphi(\lambda_i)|<\infty$.
Then~\eqref{eq:double-orth} and \cref{lem:energy-displacement} give, for $\varphi\in\{1,x\}$,
$$
|\langle D_k,\varphi\rangle_{w_{k+2}}|
\leq\frac{B_P}{H_kH_{k+1}}
 \sum_iw_{k,i}\Delta_k(\lambda_i)^2
=\frac{B_P}{H_k}d_k
\leq\frac{B_P}{H_0}d_k.
$$
Thus $\sum_{\varphi\in\{1,x\}}|\langle D_k,\varphi\rangle_{w_{k+2}}|\le Cd_k$. If $D_k(x)=D_{k,0}+D_{k,1}x$, then
$$
\binom{D_{k,0}}{D_{k,1}}
=\Gamma(w_{k+2})^{-1}
\binom{\langle D_k,1\rangle_{w_{k+2}}}{\langle D_k,x\rangle_{w_{k+2}}}.
$$
Uniform inversion proves the first bound in~\eqref{eq:factor-bounds}.

The exact factorization $q_{k+1}-q_k=P_{k+1}(P_{k+2}-P_k)$ and boundedness of $P_{k+1}$ prove the second bound in~\eqref{eq:factor-bounds}. Since $\sum_kd_k<\infty$, the even and odd quadratic sequences and the full quartic sequence are Cauchy in coefficient space, proving~\eqref{eq:even-odd-P} and the convergence assertion in~\eqref{eq:q-limit-tail}. Finally,
$$
\|q_k-q_\infty\|
\le\sum_{j=k}^{\infty}\|q_{j+1}-q_j\|
\le C\sum_{j=k}^{\infty}d_j
=C(h-H_k),
$$
which proves the remaining estimate in~\eqref{eq:q-limit-tail}.
\end{proof}

The coefficient limits now determine one fixed spectral set containing every limiting support. Define
\begin{equation}
E:=\{\lambda_i:q_\infty(\lambda_i)=h\}.
\label{eq:E}
\end{equation}
The support of every weight $\omega$-limit point is contained in $E$. To see this, choose $w^*\in\Omegaew\cup\Omegaow$ and a convergent subsequence $w_{k_j}\to w^*$ whose indices are all even or all odd. After passing to a further subsequence, compactness gives $y_{k_j}\to y^*$, where $(y_i^*)^2=w_i^*$. By continuity, $y_{k_j+1}\to F(y^*)$, and hence $P_{k_j}\to P_{y^*}$ and $P_{k_j+1}\to P_{F(y^*)}$. Thus, for every $\lambda_i\in\supp w^*=\supp y^*$,
$$
q_\infty(\lambda_i)
=\lim_{j\to\infty}q_{k_j}(\lambda_i)
=P_{y^*}(\lambda_i)P_{F(y^*)}(\lambda_i)
=h,
$$
where the last equality is \cref{prop:localization}. Hence $\supp w^*\subset E$. Since $q_\infty-h$ is a monic quartic, $|E|\leq4$.

\section{Limiting support geometry}

This section separates the two remaining cases: three-node limits and four-node limits. In the four-node case, the exact two-cycles form a one-parameter family whose endpoints have only three active nodes.

\subsection{Lagrange interpolation identities}

Here $S$ denotes a local support and $R_S(x)=\prod_{\xi\in S}(x-\xi)$, whereas $E$ is the fixed resonant set from~\eqref{eq:E}. Every limiting support lies in $E$; if a four-node limit exists, $E$ is the fixed frame and its endpoints have supports $E\setminus\{\mu\}$.

For distinct nodes $S=\{\xi_1,\ldots,\xi_m\}$, Lagrange interpolation and comparison of leading coefficients give
\begin{equation}
\sum_{\xi\in S}\frac{f(\xi)}{R_S'(\xi)}=\coeff{m-1}{f}
\qquad(\deg f\le m-1).
\label{eq:lagrange-coefficient}
\end{equation}
Thus the sum vanishes for $\deg f\le m-2$ and equals one for monic $f$ of degree $m-1$. In particular, every vector annihilating $1$ and $x$, meaning $\sum_\xi a_\xi=\sum_\xi\xi a_\xi=0$, has the form
\begin{equation}
a_\xi=\frac{\chi}{R_S'(\xi)} \qquad (|S|=3),
\label{eq:three-annihilator}
\end{equation}
or, for $E=\{e_1,e_2,e_3,e_4\}$ and $R=\prod_i(x-e_i)$,
\begin{equation}
a_i=\frac{\zeta_1e_i+\zeta_0}{R'(e_i)} \qquad (E=\{e_1,e_2,e_3,e_4\}),
\label{eq:four-annihilator}
\end{equation}
Indeed,~\eqref{eq:lagrange-coefficient} supplies the displayed spanning vectors and the nullspaces have dimensions one and two.

\subsection{The three-node case}

Exact two-block invariance on three active nodes is classical~\cite[Thm.~4.8]{Forsythe1968} and~\cite[Thm.~3]{PronzatoWynnZhigljavsky2009}. Section~1.4.1 of~\cite{PronzatoWynnZhigljavsky2009} also analyzes local uniqueness of a three-point limiting measure. Here the interpolation formula identifies the weights directly from the limiting factor. If every $\omega$-limit point has three active eigenvalues, then within each of the even and odd limit sets there is at most one weight vector for each support, and connectedness forces convergence.

\begin{proposition}\label{prop:three-node}
Let $w$ be a three-node $\omega$-limit point of either the even or the odd weight sequence, let $S=\supp w$, let $P$ be the corresponding limiting quadratic, and define $R_S(x)=\prod_{\xi\in S}(x-\xi)$. Then $P=P_w$, $H(w)=h$, and
\begin{equation}
w_\xi P(\xi)=\frac{h}{R_S'(\xi)}
\qquad(\xi\in S).
\label{eq:three-bary}
\end{equation}
In particular, $S$ and $P$ determine $w$ uniquely. If every point of $\Omegaew\cup\Omegaow$ has three active nodes, then both the even and odd weight sequences converge.
\end{proposition}

\begin{proof}
Choose a subsequence $w_{k_j}\to w$ from the relevant even or odd weight sequence. Continuity and~\eqref{eq:even-odd-P} give $P=P_w$ and $H(w)=h$. Orthogonality shows that $a_\xi=w_\xi P(\xi)$ annihilates $1$ and $x$, so~\eqref{eq:three-annihilator} gives $w_\xi P(\xi)=\chi/R_S'(\xi)$. Multiplying by $P(\xi)$ and summing, the left-hand side becomes $H(w)=h$. Since $P$ is monic quadratic,~\eqref{eq:lagrange-coefficient} gives $\sum_{\xi\in S}P(\xi)/R_S'(\xi)=1$, and hence $\chi=h$. Moreover, $h>0$, $R_S'(\xi)\neq0$, and $w_\xi>0$ imply $P(\xi)\neq0$, so~\eqref{eq:three-bary} determines every $w_\xi$.

If every limit point has three active nodes, then only finitely many supports can occur in either the even or the odd $\omega$-limit set, and each support determines at most one weight vector. Each limit set is therefore finite and, by the compactness argument above, connected; hence each is a singleton.
\end{proof}

\subsection{Parameterization of the four-node two-cycle family}

It remains to consider the case in which an $\omega$-limit point has four active eigenvalues. Since every limiting support is contained in $E$ and $|E|\leq4$, the set $E$ then consists of precisely four nodes. Zhuk constructed two-step direction-invariant states on 3 and 4 nodes and gave a sufficient local-attraction condition~\cite{Zhuk1982}. Pronzato, Wynn, and Zhigljavsky identified the one-dimensional four-support invariant geometry~\cite[secs.~1.4.1--1.4.2]{PronzatoWynnZhigljavsky2009}. Here the limiting factorization selects a fixed, not necessarily symmetric, four-node set and a closed affine family of exact two-cycles, including its three-node endpoints. We derive its scalar parameter and a signed defect that measures departure from the family.

Since $q_\infty-h$ is monic of degree four and vanishes at every point of $E$, it equals the node polynomial of $E$. Relabel the four nodes so that $e_1<e_2<e_3<e_4$, and set
$$
E=\{e_1,e_2,e_3,e_4\},
\qquad
R(x)=\prod_{i=1}^{4}(x-e_i),
\qquad
q_\infty=R+h.
$$
The limiting factors satisfy $\Pe\Po=R+h$.

To make the support restriction explicit, define the four-node weight simplex
$$
\mathcal S_E
:=\left\{w\in\mathbb R^N:
w_i\geq0,\quad \sum_{i=1}^Nw_i=1,\quad
w_i=0\ \text{whenever }\lambda_i\notin E\right\}.
$$
Thus $w\in\mathcal S_E$ means exactly that $\supp w\subseteq E$, or equivalently that all weights outside $E$ vanish. The relative interior of $\mathcal S_E$ consists of states with all four $E$-weights positive, while its boundary contains the three-node states used below. Since $T$ preserves zero coordinates, it maps every state in $\mathcal S_E$ at which it is defined back into $\mathcal S_E$.

Orthogonality on four nodes leaves a two-dimensional space of annihilating vectors. The energy identity fixes the scale of one component, leaving a single normalized parameter. Let $w\in\mathcal S_E$ have at least three positive coordinates, and set $P=P_w$ and $H=H(w)$. Since $a_i=w_{e_i}P(e_i)$ annihilates $1$ and $x$,~\eqref{eq:four-annihilator} gives $a_i=(\zeta_1e_i+\zeta_0)/R'(e_i)$. Multiplication by $P(e_i)$ and summation, followed by~\eqref{eq:lagrange-coefficient}, gives $\zeta_1=H>0$. Writing $\alpha=\alpha(w):=-\zeta_0/H$, we obtain
\begin{equation}
w_{e_i}P(e_i)=H\frac{e_i-\alpha}{R'(e_i)},
\qquad 1\le i\le4.
\label{eq:four-bary}
\end{equation}
For fixed $P$ and $H$, the right-hand side of~\eqref{eq:four-bary} is affine in $\alpha$. If exactly one weight on $E$, say $w_{e_j}$, is zero, then~\eqref{eq:four-bary} and $R'(e_j)\neq0$ give $\alpha=e_j$. Thus, on the limiting family constructed below, $\alpha$ is an affine coordinate that remains meaningful at its three-node endpoints: it records exactly which node has disappeared.

Endpoint regularity will be needed below, so we record a formula for $\alpha$ that does not divide by an individual weight. Let $\sigma_1=e_1+\cdots+e_4$. Wherever $P_w$ is defined and $H(w)>0$, set
\begin{equation}\label{eq:alpha-rational}
\widetilde\alpha(w)
=\sigma_1-\frac{\sum_{i=1}^4e_i^3w_{e_i}P_w(e_i)}{H(w)}.
\end{equation}
For $w\in\mathcal S_E$ with at least three positive weights, multiplying~\eqref{eq:four-bary} by $e_i^3$ and summing gives $\widetilde\alpha(w)=\alpha(w)$. Indeed, Lagrange coefficient extraction gives
$$
\sum_i\frac{e_i^3}{R'(e_i)}=1,
\qquad
\sum_i\frac{e_i^4}{R'(e_i)}=\sigma_1,
$$
where the second identity follows by reducing $x^4$ modulo $R$. At any $w\in\mathcal S_E$ with at least three positive weights, the moment Gram determinant and the energy are positive. Formula~\eqref{eq:alpha-rational} therefore makes $\widetilde\alpha$ a rational, hence $C^1$, function of the moments near that state. Since $\widetilde\alpha=\alpha$ on the part of $\mathcal S_E$ where $\alpha$ was defined, it provides a $C^1$ extension of $\alpha$ through each three-node endpoint encountered below.

\subsubsection{The limiting family}

We now identify the closed family of exact two-cycles selected by the limiting factors $\Pe,\Po$ and the limiting energy $h$, and determine the admissible interval for its affine parameter. At an even $\omega$-limit point, $P=\Pe$ and $H=h$, so~\eqref{eq:four-bary} gives the even formula below. The odd formula is its anticipated image under $T$, using $\Pe(e_i)\Po(e_i)=h$; the lemma verifies this relation and determines which values of $\alpha$ produce nonnegative weights:
$$
w_{e_i}^{\mathrm e}(\alpha):=\frac{h(e_i-\alpha)}{R'(e_i)\Pe(e_i)},
\qquad
w_{e_i}^{\mathrm o}(\alpha):=\frac{h(e_i-\alpha)}{R'(e_i)\Po(e_i)},
\qquad 1\le i\le4,
$$
for $\alpha\in\mathbb R$, with all other coordinates zero.

\begin{lemma}\label{lem:edge-family}
Assume that a four-node $\omega$-limit point exists. The vectors $w^{\mathrm e}(\alpha)$ and $w^{\mathrm o}(\alpha)$ sum to one for every real $\alpha$. The values of $\alpha$ for which $w^{\mathrm e}(\alpha)\ge0$ form a nondegenerate closed interval $I$. For $\alpha\in\operatorname{int}I$, all four weights are positive; at each endpoint exactly one weight vanishes. For every $\alpha\in I$, $P_{w^{\mathrm e}(\alpha)}=\Pe$ and $H(w^{\mathrm e}(\alpha))=h$. Moreover, $Tw^{\mathrm e}(\alpha)=w^{\mathrm o}(\alpha)$ and $Tw^{\mathrm o}(\alpha)=w^{\mathrm e}(\alpha)$.
\end{lemma}

\begin{proof}
At every node of $E$, $\Pe(e_i)\Po(e_i)=h$.
The partial-fraction expansion
$$
\frac{\Po(x)}{R(x)}
=\sum_{i=1}^{4}\frac{\Po(e_i)}{R'(e_i)(x-e_i)}
$$
and comparison at infinity give $\sum_i\Po(e_i)/R'(e_i)=0$ and $\sum_i e_i\Po(e_i)/R'(e_i)=1$. Since $h/\Pe(e_i)=\Po(e_i)$,
$$
\sum_iw_{e_i}^{\mathrm e}(\alpha)
=\sum_i\frac{(e_i-\alpha)\Po(e_i)}{R'(e_i)}=1.
$$
Reversing $\Pe$ and $\Po$ gives the odd normalization. Moreover, $w_{e_i}^{\mathrm o}(\alpha)=\Pe(e_i)^2w_{e_i}^{\mathrm e}(\alpha)/h$. Thus the two families have identical coordinate signs and zero sets, and their admissible values of $\alpha$ coincide.

Each coordinate $w_{e_i}^{\mathrm e}(\alpha)$ is affine in $\alpha$, with nonzero slope $-h/(R'(e_i)\Pe(e_i))$. Let
$$
I:=\{\alpha\in\R:w^{\mathrm e}(\alpha)\geq0\}.
$$
Each constraint defining $I$ is a closed half-line. A four-node weight $\omega$-limit point supplies a value $\alpha_0$ for which all four inequalities are strict, so $\alpha_0\in\operatorname{int}I$. The four slopes sum to zero because $\sum_iw_{e_i}^{\mathrm e}(\alpha)=1$; since none is zero, slopes of both signs occur. Hence $I$ is a nondegenerate compact interval. A nonzero-slope affine function that is nonnegative on $I$ cannot vanish in $\operatorname{int}I$, since it would change sign there. Thus all four weights are strictly positive in $\operatorname{int}I$. At either endpoint of $I$, at least one constraint is active. Because $w_{e_i}^{\mathrm e}(\alpha)=0$ exactly when $\alpha=e_i$, and the nodes are distinct, exactly one weight vanishes at each endpoint. In particular, no $e_i$ lies in $\operatorname{int}I$, so the endpoints are consecutive nodes: $I=[e_j,e_{j+1}]$ for some $j\in\{1,2,3\}$.

We next locate the admissible interval relative to the roots of the limiting factors. Besides describing the endpoints, this will provide the uniform nonvanishing bound needed later when ratios involving $\Pe$ and $\Po$ are formed. For $\alpha\in\operatorname{int}I$, positivity and $\Pe(e_i)\Po(e_i)=h>0$ give
$$
\operatorname{sgn}\Pe(e_i)
=\operatorname{sgn}\Po(e_i)
=\operatorname{sgn}\frac{e_i-\alpha}{R'(e_i)}.
$$
Since $\operatorname{sgn}R'(e_i)=(-,+,-,+)$, the node-sign patterns for $I=[e_1,e_2]$, $[e_2,e_3]$, and $[e_3,e_4]$ are, respectively, $(+,+,-,+)$, $(+,-,-,+)$, and $(+,-,+,+)$. In each case the two sign changes occur across the two gaps disjoint from $I$. The intermediate value theorem therefore places, for each of $\Pe$ and $\Po$, one zero in each of those gaps; since each polynomial has degree two, these are its only zeros. Thus neither $\Pe$ nor $\Po$ vanishes in $\operatorname{int}I$, and neither vanishes at an endpoint because $\Pe(e_i)\Po(e_i)=h$. Compactness gives
\begin{equation}\label{eq:factor-gap}
\ell_0:=\min_{\alpha\in I}
\min\{|\Pe(\alpha)|,|\Po(\alpha)|\}>0.
\end{equation}

Next, $w_{e_i}^{\mathrm e}(\alpha)\Pe(e_i)=h(e_i-\alpha)/R'(e_i)$. By~\eqref{eq:lagrange-coefficient}, the right side annihilates $1$ and $x$, so $\Pe$ is the monic orthogonal quadratic for these weights. Multiplying by $\Pe(e_i)$ and summing gives
$$
H(w^{\mathrm e}(\alpha))
=h\sum_i\frac{(e_i-\alpha)\Pe(e_i)}{R'(e_i)}=h,
$$
because $(x-\alpha)\Pe(x)$ is cubic and monic. Finally,
$$
(Tw^{\mathrm e}(\alpha))_{e_i}
=\frac{w_{e_i}^{\mathrm e}(\alpha)\Pe(e_i)^2}{h}
=\frac{(e_i-\alpha)\Pe(e_i)}{R'(e_i)}
=\frac{h(e_i-\alpha)}{R'(e_i)\Po(e_i)}
=w_{e_i}^{\mathrm o}(\alpha).
$$
Interchanging $\Pe$ and $\Po$ gives $Tw^{\mathrm o}(\alpha)=w^{\mathrm e}(\alpha)$.
\end{proof}

Define the even and odd line segments
$$
\Ce=\{w^{\mathrm e}(\alpha):\alpha\in I\},
\qquad
\Co=T\Ce=\{w^{\mathrm o}(\alpha):\alpha\in I\},
\qquad \cC=\Ce\cup\Co.
$$
At an even or odd $\omega$-limit point $w$, continuity and~\eqref{eq:even-odd-P} give $P_w=\Pe$ or $\Po$, respectively, and $H(w)=h$. Thus~\eqref{eq:four-bary} places $w$ in $\Ce$ or $\Co$, proving $\Omegaew\subset\Ce$ and $\Omegaow\subset\Co$. A three-node limit point is an endpoint because its missing node equals $\alpha(w)$. Since every limit point of $(w_{2n})$ lies in $\Ce$ and every limit point of $(w_{2n+1})$ lies in $\Co$, the approach-to-$\omega$-limit observation above yields $\dist(w_{2n},\Ce)\to0$ and $\dist(w_{2n+1},\Co)\to0$.
\subsection{Exact dynamics on the four-node weight simplex}

The preceding relations show that the even and odd sequences approach the two-cycle segments, but they do not prevent indefinite motion along those segments. Moreover, the actual trajectory may retain weights outside $E$. We first derive the exact dynamics for an orbit whose states remain in the invariant simplex $\mathcal S_E$; geometric decay will later control the discrepancy from the full trajectory.

For this subsection, consider an exact orbit $w_{k+1}=Tw_k$ with $w_k\in\mathcal S_E$ and at least three positive coordinates at every step. The choices below isolate the single scalar responsible for motion within $\mathcal S_E$. Since $q_k$ and $R$ are both monic quartics, $q_k-R$ has degree at most three, and $c_k$ is its leading coefficient. Meanwhile, $K_\alpha$ is the unique cubic polynomial satisfying $R(x)-R(\alpha)=(x-\alpha)K_\alpha(x)$ that is also monic. Let $\alpha_k=\alpha(w_k)$, and set
$$
\tau_k(x)=x-\alpha_k,
\qquad
K_\alpha(x)=\frac{R(x)-R(\alpha)}{x-\alpha},
\qquad
c_k=\coeff{3}{(q_k-R)}.
$$
The polynomial $K_\alpha$ is cubic and monic.

\begin{lemma}\label{lem:face-equations}
Along an exact orbit in $\mathcal S_E$ having at least three positive coordinates at every step,
\begin{align}
q_k&=R+H_{k+1}+c_kK_{\alpha_k},\qquad
\alpha_{k+1}-\alpha_k=\frac{c_kR(\alpha_k)}{H_{k+1}},
\label{eq:face-motion}\\
d_k&=\frac{c_k^2}{H_{k+1}}
\sum_{i=1}^{4}w_{k,e_i}K_{\alpha_k}(e_i)^2. \label{eq:face-diss}
\end{align}
\end{lemma}

\begin{proof}
Applying~\eqref{eq:four-bary} at time $k$, followed by one application of $T$ and the same identity at time $k+1$, gives $\tau_k(e_i)q_k(e_i)=H_{k+1}\tau_{k+1}(e_i)$ at every node of $E$. Hence the polynomial $\tau_kq_k-H_{k+1}\tau_{k+1}$ vanishes at all four roots of $R$, so it is divisible by $R$. Its quotient is monic linear, hence equals $\tau_k+c_k$; comparison of the degree-four coefficients identifies the constant with $c_k=\coeff{3}{(q_k-R)}$. Thus
\begin{equation}
\tau_k(q_k-R)=c_kR+H_{k+1}\tau_{k+1}.
\label{eq:face-A}
\end{equation}
Evaluating at $x=\alpha_k$ gives the second identity in~\eqref{eq:face-motion}.

Now use $R(x)-R(\alpha_k)=\tau_k(x)K_{\alpha_k}(x)$ and rewrite $H_{k+1}\tau_{k+1}=H_{k+1}\tau_k-c_kR(\alpha_k)$ by the second identity in~\eqref{eq:face-motion}. Substituting into~\eqref{eq:face-A} and canceling the polynomial factor $\tau_k$ gives the first identity in~\eqref{eq:face-motion}. Finally, at each node $e_i\in E$, $R(e_i)=0$, so $q_k(e_i)-H_{k+1}=c_kK_{\alpha_k}(e_i)$. Substitution into the global variance identity~\eqref{eq:variance} gives~\eqref{eq:face-diss}.
\end{proof}

Inside the two-cycle family $c_k=0$. Away from it, the second identity in~\eqref{eq:face-motion} shows that the sign of $c_kR(\alpha_k)$ determines the direction of motion in $\alpha_k$, whereas~\eqref{eq:face-diss} shows that $c_k^2$ determines the energy increment. The interval $I$ lies between two consecutive roots of $R$, so $R$ is strictly positive throughout $\operatorname{int}I$ or strictly negative throughout $\operatorname{int}I$, and it vanishes at the two endpoints. We show below that the relevant nearby states keep their $\alpha$-values in this same interval. Thus, once the next identity shows that $c_k$ cannot reverse sign, the drift in $\alpha_k$ has a consistent orientation. The first identity in~\eqref{eq:face-motion} also shows that this one scalar accounts for the entire nonconstant discrepancy of $q_k$ from $R+H_{k+1}$.

\subsection{Sign transport for the defect}

To control cancellation in the accumulated drift, we exploit the factor $P_{k+1}$ shared by consecutive product quartics. Evaluating their difference at the two roots of this factor removes the polynomial term containing $P_{k+1}$, while subtracting the resulting equations removes the scalar energy term. The divided difference below is precisely the coefficient that remains. For a monic quadratic $P$ with distinct roots $z_1,z_2$, define
$$
L(\alpha,P):=\frac{K_\alpha(z_1)-K_\alpha(z_2)}{z_1-z_2}.
$$
If $R(x)=x^4+r_3x^3+r_2x^2+r_1x+r_0$ and $P(x)=x^2+p_1x+p_0$, then
$$
K_\alpha(x)=x^3+(r_3+\alpha)x^2+(r_2+r_3\alpha+\alpha^2)x+\text{constant}.
$$
Since $z_1+z_2=-p_1$ and $z_1z_2=p_0$, direct substitution gives
\begin{equation}\label{eq:L-coeff}
L(\alpha,P)
=p_1^2-p_0-(r_3+\alpha)p_1+(r_2+r_3\alpha+\alpha^2).
\end{equation}
Equivalently, if $P(x)=x^2-\sigma x+\varpi$, then
$$
L(\alpha,P)=\sigma^2-\varpi+(r_3+\alpha)\sigma+r_2+r_3\alpha+\alpha^2.
$$
We henceforth use~\eqref{eq:L-coeff} to define $L$ for every monic quadratic, including one with a double root. Thus $L$ is polynomial in $\alpha$ and in the coefficients of $P$; no root choice or root-separation bound is involved.

\begin{lemma}\label{lem:factor-sharing}
For an exact orbit $w_{k+1}=Tw_k$ in $\mathcal S_E$ whose states have at least three positive coordinates,
\begin{equation}
c_{k+1}L(\alpha_{k+1},P_{k+1})
=c_kL(\alpha_k,P_{k+1}).
\label{eq:orientation}
\end{equation}
On the two-cycle family $\cC$,
\begin{equation}
L(\alpha,\Po)=\Pe(\alpha)\neq0,
\qquad L(\alpha,\Pe)=\Po(\alpha)\neq0.
\label{eq:L-edge}
\end{equation}
Hence there is a neighborhood $\mathcal U_E$ of $\cC$, relative to $\mathcal S_E$, such that both $u$ and $Tu$ have at least three positive coordinates for every $u\in\mathcal U_E$. On $\mathcal U_E$, the two factors
$$
L(\alpha(u),P_{Tu}),
\qquad
L(\alpha(Tu),P_{Tu})
$$
are nonzero and have the same sign; these factors, their reciprocals, and their first derivatives are uniformly bounded.
\end{lemma}

\begin{proof}
Subtract the first identity in~\eqref{eq:face-motion} at times $k$ and $k+1$:
$$
c_{k+1}K_{\alpha_{k+1}}-c_kK_{\alpha_k}
+(H_{k+2}-H_{k+1})
=P_{k+1}(P_{k+2}-P_k).
$$
Evaluate at the two roots of $P_{k+1}$, subtract the resulting equations, and divide by the root difference. This gives~\eqref{eq:orientation}.

To verify~\eqref{eq:L-edge} from~\eqref{eq:L-coeff}, express the factors as $\Pe(x)=x^2+a_{\mathrm e}x+b_{\mathrm e}$ and $\Po(x)=x^2+a_{\mathrm o}x+b_{\mathrm o}$. Comparison of the cubic and quadratic coefficients in $\Pe\Po=R+h$ gives $r_3=a_{\mathrm e}+a_{\mathrm o}$ and $r_2=b_{\mathrm e}+b_{\mathrm o}+a_{\mathrm e}a_{\mathrm o}$. Substitution into~\eqref{eq:L-coeff}, with $P=\Po$, yields $L(\alpha,\Po)=\alpha^2+a_{\mathrm e}\alpha+b_{\mathrm e}=\Pe(\alpha)$.
The reversed identity follows in the same way. The lower bound~\eqref{eq:factor-gap} therefore applies on the entire two-cycle family. Every state in $\cC$, as well as its image under $T$, has at least three positive coordinates. Continuity and compactness therefore give a sufficiently small neighborhood $\mathcal U_E$ of $\cC$ relative to $\mathcal S_E$ on which the support-size assertion and the uniform bounds hold. On $\cC$, the two displayed factors coincide; after shrinking $\mathcal U_E$, they remain nonzero and have the same sign.
\end{proof}

Equation~\eqref{eq:orientation} gives $c_{k+1}=L(\alpha_k,P_{k+1})c_k/L(\alpha_{k+1},P_{k+1})$, and the multiplier is positive near the two-cycle family. Thus the dynamics can change the size of $c_k$ but cannot reverse its sign there.

\section{Decay and convergence}

The preceding identities describe motion within the invariant simplex $\mathcal S_E$, including at its three-node endpoints. To complete the proof we must control two effects: mass outside $E$, and possible drift along the one-parameter family $\cC$. The coefficient in~\eqref{eq:face-diss} can degenerate at an endpoint, so we treat the external mass first, using the quartic tail estimate~\eqref{eq:q-limit-tail}, which never divides by an individual weight. Once that mass decays geometrically, the sequence obtained by retracting the full orbit onto $\mathcal S_E$ is a pseudo-orbit in $\mathcal S_E$ whose stepwise errors have finite total norm.

\subsection{Classification of the external multipliers}

For each external node $\lambda\in\{\lambda_1,\dots,\lambda_N\}\setminus E$, define $\rho_\lambda:=q_\infty(\lambda)/h$. This is the limiting signed two-block multiplier at $\lambda$; the corresponding multiplier for the squared weight is $\rho_\lambda^2$. This transverse-multiplier viewpoint also underlies the stability analysis in~\cite[sec.~1.4.3]{PronzatoWynnZhigljavsky2009}. The neutral case $\rho_\lambda=-1$, which is invisible after squaring the coordinates, is the case that requires the phase argument below.

Every external coordinate tends to zero. Indeed, if $w_{k,\lambda}\not\to0$ for some $\lambda\notin E$, a subsequence on which this coordinate is bounded below has an $\omega$-limit point whose support is not contained in $E$, contradicting the support inclusion established after~\eqref{eq:E}. If the coordinate is never annihilated, the two-block recurrence and~\eqref{eq:q-limit-tail} give
\begin{equation}
\frac{w_{k+2,\lambda}}{w_{k,\lambda}}
=\frac{q_k(\lambda)^2}{H_kH_{k+1}}
\to\rho_\lambda^2.
\label{eq:external-ratio}
\end{equation}
When $|\rho_\lambda|>1$, this ratio is eventually bounded below by a constant greater than one, so along either the even or the odd subsequence the coordinate would increase geometrically, contradicting its convergence to zero. Moreover, $\rho_\lambda=1$ would imply $q_\infty(\lambda)=h$, hence $\lambda\in E$. Thus every external coordinate that is never annihilated satisfies $|\rho_\lambda|<1$ or $\rho_\lambda=-1$.

\subsection{Exclusion of the multiplier \texorpdfstring{$-1$}{-1} and geometric external decay}

The only unresolved case is $\rho_\lambda=-1$. Then the signed coordinate reverses asymptotically, but its weight has limiting multiplier one, so quartic convergence gives no contraction. We therefore seek a quantity whose increment has a definite sign, yet which tends to $-\infty$ if the external weight tends to zero along an interior four-node subsequence. The term $\log w_{2n,\lambda}$ has the latter property. Adding an interpolating combination of the four frame logarithms makes the leading cubic perturbation cancel in its increment: the remaining linear term acquires an extra energy-deficit factor and is therefore quadratic compared with the positive energy drift.

Fix an external node $\lambda$ with $\rho_\lambda=-1$, equivalently $R(\lambda)=-2h$. Define the degree-three Lagrange coefficients $\Lambda_i(\lambda):=R(\lambda)/((\lambda-e_i)R'(e_i))$. They satisfy $\sum_i\Lambda_i(\lambda)=1$ and $f(\lambda)=\sum_i\Lambda_i(\lambda)f(e_i)$ whenever $\deg f\le3$.

A four-node limit point is an interior point of its segment, and $T$ carries it to an interior point of the other segment. We analyze the even subsequence and write $H_n^{\mathrm e}:=H_{2n}$, $H_n^{\mathrm o}:=H_{2n+1}$, and $q_n^{\mathrm e}:=q_{2n}$. Because zero weights are absorbing, the existence of arbitrarily late interior four-node iterates implies that every frame weight $w_{k,e_i}$ has been positive throughout. Only annihilation of the external coordinate can make $J_{\lambda,n}$ undefined.
Whenever the five relevant masses are positive, define
\begin{equation}
J_{\lambda,n}
:=\log w_{2n,\lambda}
+\sum_{i=1}^{4}\Lambda_i(\lambda)\log w_{2n,e_i}.
\label{eq:J}
\end{equation}

The coefficients $\Lambda_i(\lambda)$ need not be positive; they are chosen for exact degree-three interpolation, not convexity.

\begin{proposition}\label{prop:phase}
If $\Omegaew$ or $\Omegaow$ contains a four-node point, then, for every external node $\lambda$ with $\rho_\lambda=-1$, there is $n_0$ such that
$$
J_{\lambda,n+1}-J_{\lambda,n}\geq0
$$
whenever $n\geq n_0$ and $w_{2n,\lambda}w_{2n+2,\lambda}>0$. In particular, if the $\lambda$-coordinate is never annihilated, then $(J_{\lambda,n})$ is eventually nondecreasing.
Consequently, every external coordinate with $\rho_\lambda=-1$ is annihilated after finitely many blocks.
\end{proposition}

\begin{proof}
Fix $n$ such that both $J_{\lambda,n}$ and $J_{\lambda,n+1}$ are defined. Set $\varepsilon_n=q_n^{\mathrm e}-(R+H_n^{\mathrm o})$. This is the error after removing the limiting quartic shape and recentering at the current odd energy; it has degree at most three, exactly the range controlled by the interpolation formula. For each of the five relevant coordinates, the exact two-block recurrence may then be divided and logged:
$$
\frac{w_{2n+2,j}}{w_{2n,j}}
=\frac{q_n^{\mathrm e}(\lambda_j)^2}{H_n^{\mathrm e}H_n^{\mathrm o}}.
$$
At the nodes of $E$, $q_n^{\mathrm e}(e_i)=H_n^{\mathrm o}+\varepsilon_n(e_i)$. At the node $\lambda\notin E$,
$$
q_n^{\mathrm e}(\lambda)
=-2h+H_n^{\mathrm o}+\varepsilon_n(\lambda)
=-\bigl(2h-H_n^{\mathrm o}-\varepsilon_n(\lambda)\bigr).
$$
For all sufficiently large $n$, $q_n^{\mathrm e}(e_i)>0$ for all $i$ and $q_n^{\mathrm e}(\lambda)<0$. Thus the squared recurrence contributes $2\log|q_n^{\mathrm e}|$ at each node. Separating the constant parts $H_n^{\mathrm o}$ and $2h-H_n^{\mathrm o}$, and then using $\sum_i\Lambda_i(\lambda)=1$ and the interpolation identity $\varepsilon_n(\lambda)=\sum_i\Lambda_i(\lambda)\varepsilon_n(e_i)$, gives
\begin{equation}
J_{\lambda,n+1}-J_{\lambda,n}
=2\log\frac{2h-H_n^{\mathrm o}}{H_n^{\mathrm e}}+2\mathfrak r_n.
\label{eq:J-inc}
\end{equation}
The first term is the positive energy drift because $H_n^{\mathrm e}\leq H_n^{\mathrm o}\leq h$. Interpolation forces the linear part of the remainder to carry an additional factor $h-H_n^{\mathrm o}$, so it is quadratic relative to that drift. More precisely,

$$
\mathfrak r_n
=\log\left(1-\frac{\varepsilon_n(\lambda)}{2h-H_n^{\mathrm o}}\right)
+\sum_i\Lambda_i(\lambda)
\log\left(1+\frac{\varepsilon_n(e_i)}{H_n^{\mathrm o}}\right).
$$
Since both $H_n^{\mathrm e}$ and $H_n^{\mathrm o}$ converge to $h>0$, while $\|\varepsilon_n\|\to0$, the quantities $H_n^{\mathrm e}$, $H_n^{\mathrm o}$, and $2h-H_n^{\mathrm o}$ are uniformly bounded away from zero for all sufficiently large $n$. Evaluation at the fixed nodes is bounded by the chosen coefficient norm, so every logarithm argument is then positive and all Taylor constants below are uniform.

The linear part of $\mathfrak r_n$ is
\begin{align*}
-\frac{\varepsilon_n(\lambda)}{2h-H_n^{\mathrm o}}
+\frac1{H_n^{\mathrm o}}\sum_i\Lambda_i(\lambda)\varepsilon_n(e_i)
&=\frac{2(h-H_n^{\mathrm o})}{H_n^{\mathrm o}(2h-H_n^{\mathrm o})}\varepsilon_n(\lambda).
\end{align*}
Taylor's theorem therefore gives, in any fixed coefficient norm on cubics, $|\mathfrak r_n|\le C((h-H_n^{\mathrm o})\|\varepsilon_n\|+\|\varepsilon_n\|^2)$. Since $q_\infty=R+h$, the tail estimate in~\eqref{eq:q-limit-tail} yields $\|\varepsilon_n\|\le\|q_n^{\mathrm e}-q_\infty\|+|h-H_n^{\mathrm o}|\le C(h-H_n^{\mathrm e})$, and hence $\mathfrak r_n=O((h-H_n^{\mathrm e})^2)$. Moreover, $2\log((2h-H_n^{\mathrm o})/H_n^{\mathrm e})\ge2\log(h/H_n^{\mathrm e})\ge2(h-H_n^{\mathrm e})/h$. Substitution in~\eqref{eq:J-inc} gives $J_{\lambda,n+1}-J_{\lambda,n}\ge2(h-H_n^{\mathrm e})/h-C(h-H_n^{\mathrm e})^2\ge0$ for all sufficiently large $n$.

Suppose that the external $\lambda$-coordinate is never annihilated. Then $J_{\lambda,n}$ is defined, and the multiplier classification above gives $w_{2n,\lambda}\to0$. Along a subsequence converging to an interior four-node $\omega$-limit point, the four weights on $E$ are bounded below. Consequently~\eqref{eq:J} gives $J_{\lambda,n}\to-\infty$, contradicting the eventual monotonicity just proved.
\end{proof}

The only external coordinates that can survive indefinitely are now strictly contractive. Since the spectrum is finite, the individual contraction estimates combine into a uniform geometric bound. For a weight vector $w$, define its external mass by $U(w):=\sum_{\lambda_i\notin E}w_i$. Under the standing assumption that a four-node limit exists, there are $C<\infty$ and $0<\theta<1$ such that
\begin{equation}
U(w_k)\le C\theta^k
\qquad(k\ge0).
\label{eq:external-decay}
\end{equation}
To prove this, consider the external nodes whose coordinates are never annihilated. The classification above and \cref{prop:phase} show that each such node satisfies $|\rho_\lambda|<1$. If there are none, every external coordinate vanishes after a common finite index, and the conclusion is immediate.

Otherwise, for each surviving external node choose $\rho_\lambda^2<\gamma_\lambda<1$. Equation~\eqref{eq:external-ratio} gives $w_{k+2,\lambda}\le\gamma_\lambda w_{k,\lambda}$ for all sufficiently large $k$, and iteration separately along the even and odd subsequences yields $w_{k,\lambda}\le C_\lambda(\sqrt{\gamma_\lambda})^k$. There are only finitely many external nodes. Taking the largest $\sqrt{\gamma_\lambda}$, and enlarging the constant to cover the finitely many initial indices and the coordinates annihilated later, proves~\eqref{eq:external-decay}. In particular,~\eqref{eq:external-decay} gives $\sum_kU(w_k)<\infty$. The endpoint retraction estimate below converts this bound on the external mass into stepwise errors relative to the dynamics in $\mathcal S_E$; the sum of their norms is finite.

\subsection{Convergence of the four-node weights}

After~\eqref{eq:external-decay}, the only possible failure of convergence is drift along the continuum $\cC$: approaching a compact segment does not by itself imply approaching one point. We retract onto $\mathcal S_E$ and track two scalars. The parameter $\alpha$ records position along the family, while $c$ is the signed defect from the exact two-cycle identity. Near $\cC$ within $\mathcal S_E$, $c$ is transported by a positive multiplier and drives $\alpha$ in a fixed orientation, while the external mass contributes only exponentially decaying errors.

The following abstract lemma records this mechanism. If $c_n$ eventually keeps one sign, the motion of $\alpha_n$ is one-sided up to exponentially decaying errors; if $c_n$ crosses zero infinitely often, positivity of $M_n$ forces $c_n$ itself to be exponentially small.

\begin{lemma}\label{lem:perturbed-orientation}
Let $\alpha_n$ lie in a compact interval and suppose
\begin{equation}
c_{n+1}=M_nc_n+\xi_n,\qquad
\alpha_{n+1}-\alpha_n=\mathcal A_nc_n+\eta_n,
\label{eq:perturbed-system}
\end{equation}
where, for some $C<\infty$ and $0<\theta<1$,
$$
M_n>0,
\qquad M_n\to1,
\qquad |\xi_n|+|\eta_n|\le C\theta^n,
\qquad |\mathcal A_n|\le C,
$$
and the nonzero $\mathcal A_n$ all have the same sign. Then $\sum_n|\alpha_{n+1}-\alpha_n|<\infty$, so $\alpha_n$ converges.
\end{lemma}

\begin{proof}
If $c_n$ is eventually nonnegative or eventually nonpositive, then $\mathcal A_nc_n$ is eventually nonnegative or eventually nonpositive. Choose $n_0$ after these sign conditions hold. Summing the second relation in~\eqref{eq:perturbed-system} from $n_0$ to $m$ gives $\sum_{n=n_0}^{m}\mathcal A_nc_n=\alpha_{m+1}-\alpha_{n_0}-\sum_{n=n_0}^{m}\eta_n$. The right side is bounded, while the partial sums on the left are monotone. Hence $\sum_{n\ge n_0}|\mathcal A_nc_n|<\infty$, and the second relation in~\eqref{eq:perturbed-system} gives finite total variation.

Otherwise there are arbitrarily late crossings $s$ with $c_sc_{s+1}\le0$. If $M_s\ge m_0>0$, then $m_0|c_s|+|c_{s+1}|\le|\xi_s|$, and hence
$$
|c_s|+|c_{s+1}|\le C\theta^s.
$$
Choose $m_0$ with $\theta<m_0<1$. For large $n$, $M_n\ge m_0$. Solving the first relation in~\eqref{eq:perturbed-system} backward from a crossing $s>n$ gives
$$
|c_n|
\le C m_0^{-(s-n)}\theta^s
+C\sum_{j=n}^{s-1}m_0^{-(j-n+1)}\theta^j
\le C'\theta^n.
$$
Letting $s\to\infty$ through crossings proves this estimate for all sufficiently large $n$. Thus $\sum_n|c_n|<\infty$, and the second relation in~\eqref{eq:perturbed-system} again gives finite total variation.
\end{proof}

At an endpoint one frame weight vanishes, so estimates based on division by individual weights are not uniform. The argument below instead uses moment determinants, energies, and total frame mass, which remain regular because every endpoint retains three active nodes. It also supplies a retraction whose two-block defect is controlled by the external mass.

Set $\mathcal G:=T^2$ so that the retracted dynamics return to the same segment of the two-cycle family.

\begin{lemma}
\label{lem:endpoint-chart}
There exist a relative neighborhood $\mathcal N$ of the compact two-cycle family $\cC$ in the probability simplex and a constant $\nu>0$ such that
$$
\det\Gamma(w)\geq\nu,\qquad H(w)\geq\nu,\qquad
\det\Gamma(Tw)\geq\nu,\qquad H(Tw)\geq\nu
$$
for every $w\in\mathcal N$. The maps $P,H,T$, and $\mathcal G=T^2$ admit $C^1$ extensions to an open neighborhood of $\overline{\mathcal N}$ in the affine hyperplane $\sum_iw_i=1$, with uniformly bounded first derivatives there. The scalar coordinates $\alpha$ and $c$, initially defined on $\mathcal U_E\subset\mathcal S_E$, where $c(u):=\coeff{3}{(P_uP_{Tu}-R)}$, admit $C^1$ extensions to the same neighborhood, including through the three-node endpoints.
\end{lemma}

\begin{proof}
Every point of $\cC$ has at least three positive weights on $E$. Consequently $\det\Gamma(w)>0$ and $H(w)>0$ on $\cC$. Since $T\cC=\cC$, the same is true of $\det\Gamma(Tw)$ and $H(Tw)$. Compactness and continuity therefore give a relative open neighborhood $\mathcal V$ of $\cC$ and a constant $\nu>0$ such that
$$
\det\Gamma(w)\geq\nu,\qquad H(w)\geq\nu,\qquad
\det\Gamma(Tw)\geq\nu,\qquad H(Tw)\geq\nu
$$
throughout $\mathcal V$. We also choose $\mathcal V$ so that $\mathcal V\cap\mathcal S_E\subset\mathcal U_E$, where $\mathcal U_E$ is the relative neighborhood supplied by \cref{lem:factor-sharing}. These quantities are precisely the denominators in the rational formulas for the maps under consideration, so $P,H,T$, and $\mathcal G$ are $C^1$ throughout $\mathcal V$.

By~\eqref{eq:alpha-rational}, $\alpha$ agrees with the $C^1$ function $\widetilde\alpha$ on $\mathcal V\cap\mathcal S_E$. Together with $\widetilde c(w)=\coeff{3}{(P_wP_{Tw}-R)}$, this provides the asserted $C^1$ extensions of $\alpha$ and $c$ through both three-node endpoints. We retain the notation $\alpha,c$ for these extensions.

Choose a relative neighborhood $\mathcal N$ of $\cC$ whose closure is contained in $\mathcal V$. By continuity of $T$ and $\mathcal G$ and the invariance
$T\cC=\mathcal G\cC=\cC$, we may also require
$$
T(\overline{\mathcal N})\cup
\mathcal G(\overline{\mathcal N})\subset\mathcal V.
$$
The denominator bounds and rational formulas provide $C^1$ extensions to an open neighborhood of $\overline{\mathcal N}$ in the affine hyperplane $\sum_iw_i=1$. The asserted derivative bounds then follow from compactness of $\overline{\mathcal N}$.
\end{proof}

We record two consequences. First, since $U=0$ on $\cC$ and $\mathcal G\cC=\cC$, after shrinking $\mathcal N$ and $\nu$ if necessary,
$$
1-U(w)\geq\nu,
\qquad
1-U(\mathcal G(w))\geq\nu
\qquad(w\in\mathcal N).
$$
For any probability vector $w$ with $U(w)<1$, define the retraction $\retE$ onto $\mathcal S_E$ by
$$
(\retE w)_i=
\begin{cases}
w_i/(1-U(w)),&\lambda_i\in E,\\
0,&\lambda_i\notin E.
\end{cases}
$$
This operation removes the weights outside $E$ and renormalizes the remaining weights. The preceding lower bounds ensure that both $\retE w$ and $\retE\mathcal G(w)$ are defined for $w\in\mathcal N$. After shrinking $\mathcal N$ once more, there is $C<\infty$ such that, for every $w\in\mathcal N$,
\begin{equation}\label{eq:retraction}
\norm{w-\retE w}_1=2U(w),\qquad
\norm{\retE\mathcal G(w)-\mathcal G(\retE w)}_1
\leq C U(w).
\end{equation}

Fix $w\in\mathcal N$ and put $v=\retE w$. For $\lambda_i\in E$,
$$
v_i-w_i=\frac{U(w)}{1-U(w)}w_i,
$$
whereas $v_i-w_i=-w_i$ for $\lambda_i\notin E$. Hence
$$
\sum_{\lambda_i\in E}|v_i-w_i|
=\frac{U(w)}{1-U(w)}\sum_{\lambda_i\in E}w_i
=U(w),
\qquad
\sum_{\lambda_i\notin E}|v_i-w_i|=U(w),
$$
and therefore $\norm{w-\retE w}_1=2U(w)$.

For the second estimate, note that every state in $\cC$ is supported on $E$, so $U(w)\leq\dist_1(w,\cC)$. If $w_t=(1-t)w+tv$ with $0\leq t\leq1$, then
$$
\dist_1(w_t,\cC)
\leq\dist_1(w,\cC)+\norm{w_t-w}_1
\leq3\dist_1(w,\cC).
$$
After shrinking $\mathcal N$, every such segment and its image under $\mathcal G$ lie in a fixed compact neighborhood of $\cC$ on which the chart maps are $C^1$ and the total $E$-mass is bounded away from zero. Consequently $\Psi=\retE\circ\mathcal G$ has a uniformly bounded derivative along the segment. Moreover, $\mathcal G$ preserves $\mathcal S_E$, so $\retE\mathcal G(v)=\mathcal G(v)$. The mean-value theorem gives
$$
\begin{aligned}
\norm{\retE\mathcal G(w)-\mathcal G(\retE w)}_1
&=\norm{\Psi(w)-\Psi(v)}_1\\
&\leq C\norm{w-v}_1
=2CU(w).
\end{aligned}
$$
This proves~\eqref{eq:retraction} uniformly through both endpoints.

A second consequence is a common positivity gap. After shrinking $\mathcal N$ once more, every $u\in\mathcal N\cap\mathcal S_E$ satisfies
\begin{equation}
\alpha(u)\in I,
\qquad
\alpha(Tu)\in I.
\label{eq:common-gap}
\end{equation}
Indeed, since $\Pe(e_i)\Po(e_i)=h>0$ for every $e_i\in E$, the node values of the two limiting factors are nonzero and have the same sign pattern. By continuity, $P_u(e_i)$ and $P_{Tu}(e_i)$ retain this pattern near $\cC$. The nonnegativity conditions in~\eqref{eq:four-bary} therefore place both parameters in the same closed gap $I$, including its endpoints. Consequently $R(\alpha(u))$ and $R(\alpha(Tu))$ are either both nonnegative or both nonpositive.

Since $\dist(w_{2n},\Ce)\to0$, the even iterates eventually enter this neighborhood. The endpoint retraction estimate~\eqref{eq:retraction} and the geometric estimate~\eqref{eq:external-decay} make the retracted sequence $(\retE w_{2n})$, whose terms lie in $\mathcal S_E$, an exponentially accurate pseudo-orbit of $\mathcal G=T^2$.

\begin{proposition}\label{prop:frame-convergence}
In the four-node case, both the even and odd subsequences of $w_k$ converge.
\end{proposition}

\begin{proof}
Work first with the even subsequence and set $\widehat w_n=\retE w_{2n}$, $\alpha_n=\alpha(\widehat w_n)$, and $c_n=c(\widehat w_n)$.
We verify the scalar lemma for these retracted states. Factor sharing supplies the positive defect multiplier, the common positivity gap fixes the sign of the $\alpha$-increment, and convergence to the limiting quartic forces the multiplier to approach one.
After restricting to a sufficiently late tail and, if necessary, replacing $\theta$ by a larger number still less than one,~\eqref{eq:retraction}, $\widehat w_{n+1}=\retE\mathcal G(w_{2n})$, and~\eqref{eq:external-decay} imply
\begin{equation}\label{eq:pseudo}
\norm{\widehat w_{n+1}-\mathcal G(\widehat w_n)}_1\leq C\theta^n.
\end{equation}

Moreover,
$$
\dist_1(\widehat w_n,\Ce)
\leq \norm{\widehat w_n-w_{2n}}_1+\dist_1(w_{2n},\Ce)
\longrightarrow0.
$$
After discarding finitely many terms, therefore, $\widehat w_n\in\mathcal N$.

The retracted iterates are now an approximate orbit in $\mathcal S_E$, where the exact scalar identities apply. Here $c=0$ characterizes the exact family. For $u\in\mathcal N\cap\mathcal S_E$, \cref{lem:factor-sharing} gives
$$
c(Tu)=m(u)c(u),
\qquad
m(u):=\frac{L(\alpha(u),P_{Tu})}{L(\alpha(Tu),P_{Tu})}>0.
$$
Applying this identity twice yields
$$
c(\mathcal G(u))=M(u)c(u),
\qquad
M(u):=m(Tu)m(u)>0.
$$
Likewise, applying the second identity in~\eqref{eq:face-motion} twice yields $\alpha(\mathcal G(u))-\alpha(u)=\mathcal A(u)c(u)$, where
$$
\mathcal A(u)=\frac{R(\alpha(u))}{H(Tu)}
+m(u)\frac{R(\alpha(Tu))}{H(\mathcal G(u))}.
$$
Let $\sigma_R\in\{-1,1\}$ denote the sign of $R$ on $\operatorname{int}I$. By~\eqref{eq:common-gap},
$$
\sigma_RR(\alpha(u))\geq0,
\qquad
\sigma_RR(\alpha(Tu))\geq0.
$$
Since $m(u)>0$ and both energy denominators are positive,
$$
\sigma_R\mathcal A(u)
=\frac{\sigma_RR(\alpha(u))}{H(Tu)}
+m(u)\frac{\sigma_RR(\alpha(Tu))}{H(\mathcal G(u))}
\geq0.
$$
Thus every nonzero value of $\mathcal A(u)$ has the same fixed sign. The uniform lower bounds for the two energy denominators, together with the uniform bounds for $m$ and $R$ on the compact neighborhood, also give $|\mathcal A(u)|\leq C$.

Equation~\eqref{eq:retraction} gives $\norm{w_{2n}-\widehat w_n}_1=O(U(w_{2n}))\to0$. Since the quartic map $q(u)=P_uP_{Tu}$ is $C^1$ by \cref{lem:endpoint-chart} and $q(w_{2n})=q_{2n}$, we obtain $q(\widehat w_n)\to q_\infty$. Because $q_\infty-R=h$ is constant, its cubic coefficient is zero, and therefore $c_n=\coeff{3}{(q(\widehat w_n)-R)}\to0$.

It remains to show that the two-step defect multiplier becomes neutral. On the exact two-cycle family, $Tu$ has the same parameter as $u$, so $m(u)=1$. The next estimates show that the small defect $c(u)$ controls the departure from this identity.
For $u\in\mathcal N\cap\mathcal S_E$, the second identity in~\eqref{eq:face-motion} gives
$$
|\alpha(Tu)-\alpha(u)|
=\frac{|c(u)R(\alpha(u))|}{H(Tu)}
\leq C|c(u)|.
$$
Writing $P=P_{Tu}$ and using the definition of $m$,
$$
m(u)-1
=\frac{L(\alpha(u),P)-L(\alpha(Tu),P)}{L(\alpha(Tu),P)}.
$$
The coefficient formula~\eqref{eq:L-coeff}, the uniform bound for $\partial_\alpha L$, and the lower bound for $|L|$ therefore imply
$$
|m(u)-1|\leq C|\alpha(Tu)-\alpha(u)|
\leq C|c(u)|.
$$
The transport identity and boundedness of $m$ give
$$
|c(Tu)|=|m(u)c(u)|\leq C|c(u)|.
$$
Because $u,Tu$, and $\mathcal G(u)$ remain in the common regularity neighborhood, the same estimate at $Tu$ yields
$$
|m(Tu)-1|\leq C|c(Tu)|\leq C|c(u)|.
$$
Consequently,
$$
\begin{aligned}
|M(u)-1|
&=|m(Tu)m(u)-1|\\
&\leq |m(Tu)-1|\,|m(u)|+|m(u)-1|\\
&\leq C|c(u)|.
\end{aligned}
$$
Since $c(\widehat w_n)=c_n\to0$, this proves quantitatively that $M(\widehat w_n)\to1$.

Applying the uniformly $C^1$ coordinates $(\alpha,c)$ to the pseudo-orbit~\eqref{eq:pseudo} gives $c_{n+1}=M(\widehat w_n)c_n+\xi_n$ and $\alpha_{n+1}-\alpha_n=\mathcal A(\widehat w_n)c_n+\eta_n$, with $|\xi_n|+|\eta_n|\le C\theta^n$. The preceding estimates verify every hypothesis of \cref{lem:perturbed-orientation}, so $\alpha_n\to\alpha_\infty$.

Once $\alpha_n$ converges, the barycentric formula recovers each frame weight separately. Indeed, the $C^1$ endpoint bounds from \cref{lem:endpoint-chart}, the convergence~\eqref{eq:even-odd-P}, and $\norm{w_{2n}-\widehat w_n}_1\to0$ give $P_{\widehat w_n}\to\Pe$ and $H(\widehat w_n)\to h$. Together with $\alpha_n\to\alpha_\infty$, the identity $\widehat w_{n,e_i}P_{\widehat w_n}(e_i)=H(\widehat w_n)(e_i-\alpha_n)/R'(e_i)$ therefore gives a limit for every weight on $E$. The limiting factor $\Pe(e_i)$ is nonzero because $\Pe(e_i)\Po(e_i)=h>0$. Since the external mass tends to zero, the even weights $w_{2n}$ converge. Their limit has at least three active nodes, so $T$ is continuous there; from $w_{2n+1}=Tw_{2n}$, the odd weights also converge.
\end{proof}

\subsection{Recovery of signed directions and completion of the proof}

Passing to weights discarded the coordinate signs, so weight convergence is not yet directional convergence. The signed two-block ratio restores precisely this missing information.

\begin{proof}[of \cref{thm:main}]
If neither $\Omegaew$ nor $\Omegaow$ contains a four-node point, \cref{prop:three-node} gives convergence of the squared spectral coordinates along both the even and odd subsequences. Otherwise, \cref{prop:frame-convergence} gives the same conclusion.

Fix either the even or the odd subsequence, and let $k$ range over its indices. If $w_{k,i}$ has a positive limit, then $\lambda_i\in E$, so $q_\infty(\lambda_i)=h$, and $y_{k,i}$ was never annihilated. Hence~\eqref{eq:signed-ratio} gives $y_{k+2,i}/y_{k,i}=q_k(\lambda_i)/\sqrt{H_kH_{k+1}}\to1$, making its sign eventually constant. Coordinates with zero limiting weight satisfy $|y_{k,i}|=\sqrt{w_{k,i}}\to0$. Thus $y_k$ converges along the chosen subsequence, and the same argument applies to the other. The limiting supports contain at least three nodes and have energy $h>0$, so $F$ is continuous there and gives $y_{\mathrm o}=F(y_{\mathrm e})$ and $y_{\mathrm e}=F(y_{\mathrm o})$.
\end{proof}

\section{Discussion}

We conclude with a computational consequence and indicate which parts of the proof may extend beyond restart length two.

\subsection{A possible extrapolation}\label{sec:extrapolation}
Directional convergence yields a vector Aitken-type postprocessing step for restarted CG with $s=2$ in exact arithmetic. Let $e_k:=x_*-x^{[k]}=A^{-1}r_k$. Since both roots of every $P_k$ lie in $[\lambda_1,\lambda_N]$, we have $P_k(0)\geq\lambda_1^2$ and hence $\Pe(0),\Po(0)>0$. The residual recurrence and polynomial convergence give
$$
\frac{\|r_{k+2}\|}{\|r_k\|}
=\frac{\sqrt{H_kH_{k+1}}}{P_k(0)P_{k+1}(0)}
\longrightarrow
\varrho:=\frac{h}{\Pe(0)\Po(0)}>0.
$$
The same value applies to both parities. The norm-ratio limit and $y_{k+2}-y_k\to0$ give the residual relation; applying $A^{-1}$ and norm equivalence gives the error relation:
$$
r_{k+2}=\varrho r_k+o(\|r_k\|),
\qquad
e_{k+2}=\varrho e_k+o(\|e_k\|).
$$
Taking $A$-norms in the second relation and using norm equivalence and the two-block Kantorovich estimate gives $0<\varrho\leq((\kappa(A)-1)/(\kappa(A)+1))^2<1$.
Consequently,
$$
\widehat\varrho_k
:=\frac{\langle r_{k+2},r_k\rangle}{\|r_k\|^2}
\longrightarrow\varrho,
\qquad
\widehat x^{[k]}
:=\frac{x^{[k+2]}-\widehat\varrho_k x^{[k]}}
{1-\widehat\varrho_k}
$$
satisfies $\|x_*-\widehat x^{[k]}\|=o(\|e_k\|)$, so it cancels the leading error along each parity subsequence. We have not studied its behavior in finite-precision arithmetic.

\subsection{The polynomial mechanism beyond two steps}

For restart length $s$, $P_w$ becomes the monic degree-$s$ polynomial orthogonal to lower-degree polynomials. Support bounds, minimal-support cycles, invariant states, and parameter limits are known in various forms~\cite{Forsythe1968,PronzatoWynnZhigljavsky2009,Zhuk1982,Zhuk1995}. Under analogous compactness and two-cycle localization hypotheses, nontermination and the limiting degree-$2s$ identity give $s+1\leq|\supp y^*|\leq2s$. Uniform moment-Gram invertibility and two-block orthogonality then suggest an analogous coefficient estimate. Directional convergence still requires new control of localization, intermediate support sizes, and neutral transverse modes.

\subsection{The use of AI assistance}
The ideas and arguments in this proof were developed and tested with substantial assistance from ChatGPT (GPT-5.6 Sol Ultra). It repeatedly proposed plausible directions for the proof, helping us identify which questions merited attention and focus our effort on learning the specific mathematics required at each stage. Its ability to scrutinize long arguments, expose hidden assumptions, and test omitted cases had a profound effect on the project. Equally striking was its broad knowledge of the literature: by pointing us toward relevant results and connections, it allowed us to read much more selectively and deeply. It was an extremely fun process to learn about areas of mathematics from this process. ChatGPT also wrote the Python code used for the figure and helped polish the exposition. The authors independently checked the code and cited sources and verified every mathematical argument; they remain fully responsible for the manuscript. Our experience suggests that the profession is certainly changing: AI can already serve not only as a writing aid, but as a serious tool for mathematical exploration, verification, and focused learning.

\section*{Acknowledgments}
M. J. C. and A. T. thank the organizers of the 2025 Simons workshop at Berkeley, where they learned of the Forsythe conjecture. A. T. acknowledges support from the Defense Advanced Research Projects Agency through The Right Space (TRS) Disruption Opportunity 25 (DARPA-PA-24-04-07) and from NSF CAREER grant DMS-2045646. G. S. acknowledges support from the NSF Mathematical Sciences Postdoctoral Research Fellowship (MSPRF)
under award number 2402074.

\phantomsection
\bibliographystyle{siamplain}
\bibliography{forsythe_references}

\end{document}